\documentclass[11pt]{amsart}
\usepackage{amssymb}
\usepackage{abraces}
\newtheorem{thm}[]{Theorem}[section]
\newtheorem{lem}{Lemma}[section]

\newtheorem{proposition}{Proposition}[section]

\newtheorem{conjecture}{Conjecture}[section]
\newtheorem*{remark}{Remark}
\newtheorem{example}{Example}[section]

\usepackage{comment}
\usepackage{tikz}
\usepackage{hyperref}

\title{A ``supernormal'' partition statistic}

\author{Madeline Locus Dawsey, Matthew Just, and Robert Schneider}

\address{Department of Mathematics\newline
University of Texas at Tyler\newline
Tyler, Texas 75799, U.S.A.}
\email{mdawsey@uttyler.edu}

\address{Department of Mathematics\newline
Emory University\newline
Atlanta, Georgia 30322, U.S.A.}
\email{mrjust@emory.edu}

\address{Department of Mathematics\newline
University of Georgia\newline
Athens, Georgia 30602, U.S.A.}
\email{robert.schneider@uga.edu}

\begin{document}

\maketitle

\begin{abstract}
We study a bijective map from integer partitions to the prime factorizations of integers that we call the ``supernorm'' of a partition, in which the multiplicities of the parts of partitions are mapped to the multiplicities of prime factors of natural numbers. 
The supernorm is connected to a family of maps we define, which suggests the potential to apply techniques from partition theory to identify and prove multiplicative properties of integers. We make a brief study of pertinent analytic aspects of the supernorm. Then, as an application of ``supernormal'' mappings (i.e., pertaining to the supernorm statistic), we prove an analogue of a formula of Kural-McDonald-Sah to give arithmetic densities of subsets of $\mathbb N$ instead of natural densities in $\mathbb P$ like previous formulas of this type; this builds on works of Alladi, Ono, Wagner, and the first and third authors. Finally, using a table of ``supernormal'' additive-multiplicative correspondences, we conjecture Abelian-type formulas that specialize to our main theorem and other known results. 
\\
  

\end{abstract}  



\section{Introduction: the supernorm statistic on partitions}\label{sect1}

In this paper, we study a statistic on integer partitions that we call the ``supernorm'', by way of which certain combinatorial and set-theoretic aspects of partition theory are mapped to the prime factorizations of integers. 

Partitions are finite multisets of natural numbers which ripple with interesting and useful set-theoretic properties. 
Let $\mathcal P$ denote the set of all integer partitions, including the empty partition $\emptyset$. We denote a generic partition by $\lambda=(\lambda_1, \lambda_2, \lambda_3, \dots, \lambda_r), \lambda_1\geq \lambda_2 \geq \dots \geq \lambda_r\geq 1$. There are a number of natural statistics on 
partitions that are of interest in number theory and combinatorics, such as the {\it size} $|\lambda|=\sum_{1\leq i \leq r}\lambda_i$ (sum of parts) and the related {\it partition function} $p(n)$ that counts the number of size-$n$ partitions, $n\geq 0$, with $|\emptyset|:=0,\   p(0):=1$; the {\it length} $\ell(\lambda)$ (number of parts) with $\ell(\emptyset):=0$; the {\it largest part} $\text{lg}(\lambda)=\lambda_1$ and {\it smallest part} $\text{sm}(\lambda)=\lambda_r$ of the partition; the {\it multiplicity} $m_i=m_i(\lambda)\geq 0$ (frequency) of the part $i\geq 1$ in $\lambda$; along with Dyson's rank, the Andrews-Garvan crank, and other well-known maps from the set of partitions $\mathcal P$ to $\mathbb N$. 
We also write partitions in ``multiplicity superscript'' notation 
$\lambda = \left<1^{m_1}, 2^{m_2}, 3^{m_3},\dots, i^{m_i},\dots \right>$ where at most finitely many $m_i\geq 0$ are nonzero (see \cite{A88} for more about partitions). 

Striking patterns abound in $\mathcal P$ much as in $\mathbb N$, and in the values of these partition statistics, one observes arithmetic properties which often require subtle methods to prove. 
In the pair of papers \cite{Robert_zeta, Robert_bracket} and subsequent works, the third author outlined a {\it multiplicative} theory of integer partitions. This additive-multiplicative theory contains a number of further partition statistics as direct analogues of classical arithmetic functions, such as an analogue of the M\"{o}bius function $\mu(n)$, viz.
\begin{equation}\label{mobius}
\mu_{\mathcal P}(\lambda)= \left\{
        \begin{array}{ll}
            0 & \text{if $\lambda$ has any part repeated,}\\
           (-1)^{\ell(\lambda)} & \text{otherwise},
        \end{array}
    \right.
\end{equation}
as well as partition analogues of the Euler phi function and other arithmetic functions, and a partition generalization of the theory of Dirichlet series. The theory is equally closely aligned with the theory of $q$-series (see \cite{Andrews_q, Berndt_q, Robert_bracket}). 

The focus in these works shifts somewhat from the usual size statistic to a multiplicative partition statistic called the {\it norm} $N(\lambda)$ (product of parts):
\begin{equation}\label{normdef}
N(\lambda)\  :=\  \lambda_1  \lambda_2  \lambda_3 \cdots  \lambda_r\  =\  1^{m_1} 2^{m_2} 3^{m_3} \cdots  i^{m_i}\cdots \in \mathbb N, 
\end{equation}
where $N(\emptyset):=1$ (it is an empty product).\footnote{For more in-depth treatments of the partition norm, see \cite{KumarRana_supernorm, Robert_bracket, SS_norm}. Partitions of fixed norm $n$ are called multiplicative partitions of $n$ (or factorizations of $n$) in the literature; these were apparently first studied by MacMahon \cite{MacMahon}. Following Oppenheim \cite{Oppenheim}, the study of these objects is sometimes referred to as ``{\it factorisatio numerorum}''. 
} 
Furthermore, defined in \cite{Robert_bracket} is a {\it partition multiplication} operation $\lambda \cdot \lambda'$ for $\lambda, \lambda' \in \mathcal P$, which is concatenation (multiset union) of the parts of $\lambda$ and $\lambda'$.  For example, if $\lambda=(3,1,1)$ and $\lambda'=(4,3,2)$, then $\lambda\cdot \lambda'=(4,3,3,2,1,1).$ This leads also to partition division and a theory of partition divisors, or {\it subpartitions}, analogous to divisors in arithmetic \cite{Robert_bracket}. 

The guiding principle behind this multiplicative theory of partitions is: {\it phenomena in classical multiplicative number theory are special cases of partition-theoretic structures.} They can be thought of as prototypes for more general theorems about partitions; conversely, known facts about partitions suggest analogous facts about numbers. 


Here we study a natural multiplicative partition statistic resembling the norm $N(\lambda)$ in \eqref{normdef} that we call the {\it supernorm} $\widehat{N}(\lambda)$: 
 \begin{equation}\label{supernormdef}
 \widehat{N}(\lambda)\  :=\  p_{\lambda_1}p_{\lambda_2}\cdots p_{\lambda_r} \  =\  2^{m_1} 3^{m_2} 5^{m_3} \cdots p_i^{m_i}\cdots \in \mathbb N,\end{equation}
where $p_i\in \mathbb P$ denotes the $i$th prime number, $i\geq 1$, and $m_i=m_i(\lambda)\geq 0$ is the multiplicity of the part $i$ in the partition $\lambda$; we define $\widehat{N}(\emptyset):=1$.\footnote{The supernorm is referred to as the {\it Heinz number} of $\lambda$ in the OEIS \cite{heinz}; Kumar gives an analytic treatment of prime partitions in \cite{KumarRana_supernorm2} that may be used to study aspects of $\widehat{N}$.}

\begin{proposition}\label{prop0}
The supernorm map $\widehat{{N}}:\mathcal P \to \mathbb  N$ 
is an isomorphism between the monoid $(\mathcal P, \  \cdot\  )$ with partition multiplication and the monoid $(\mathbb N, \  \cdot\  )$ with integer multiplication.
\end{proposition}

\begin{proof}
Let us denote the {\it index} (subscript) of the $i$th prime $p_i$ by  the function $\operatorname{idx}(p_i):=i$, and define $\operatorname{idx}(1):=0$. \footnote{Converting partition parts to indices of variables is a key technique of Sills in \cite{Sills_Macmahon}.}  

It is immediate from the fundamental theorem of arithmetic that $\widehat{N}(\lambda)$ is the unique integer whose prime factors' indices are the parts of $\lambda\in \mathcal P$. Clearly, the map $\widehat{N}:\mathcal P \to \mathbb N$ is reversible. Let us denote by the {\it pre-partition} $\widehat{P}(n)\in \mathcal P$ of $n\in \mathbb N$ the partition $\lambda$ whose supernorm $\widehat{N}(\lambda)$ is equal to $n=2^{a_1}3^{a_2}5^{a_3}\cdots p_i^{a_i}\cdots, a_i\geq 0$:
\begin{equation}\label{prepartitiondef}
 \widehat{P}(n):=\left<1^{a_1}, 2^{a_2}, 3^{a_3}, \dots, i^{a_i},  \dots \right>\in \mathcal P. \end{equation}
 The uniqueness of the prime factorization of $n$, which yields a unique multiset of prime factor indices, guarantees the uniqueness of the pre-partition $\lambda=\widehat{P}(n)$ whose parts are identically this multiset of indices. 
Thus $\widehat{N}$ is a bijective map, and we have  
$$\left( \widehat{P}\circ \widehat{N}\right)(\lambda)=\lambda,\  \  \  \  \  \  \  \  \left( \widehat{N}\circ \widehat{P}\right)(n)=n.$$
Now, the sets $\mathcal P$ and $\mathbb N$ are monoids under their respective multiplication operations. Observe for $\lambda,\lambda'\in \mathcal P$ that $\widehat{N}(\lambda\cdot \lambda')=\widehat{N}(\lambda)\widehat{N}(\lambda')$ with $\widehat{N}(\emptyset)=1$. Similarly, if we set $n:=\widehat{N}(\lambda), n':=\widehat{N}(\lambda')$, then $\widehat{P}(nn')=\widehat{P}(n)\widehat{P}(n')=\lambda\cdot \lambda'$ with $\widehat{P}(1)=\emptyset$, which is the multiplicative identity in $\mathcal P$. The bijection preserves multiplication; thus $\widehat{N}$ is an isomorphism of monoids.
\end{proof}

 Throughout this paper, we shall use the term ``supernormal'' in the sense ``pertaining to the supernorm''. 
 
 The ``supernormal'' 
isomorphism between partitions in $\mathcal P$ and prime factorizations in $\mathbb N$ can be observed immediately from the two lattices in Figure \ref{figure1}. 
Note that the sequence of multiplicities $(m_1,m_2,m_3,\dots)$ associated to the parts of $\lambda = \left<1^{m_1}, 2^{m_2}, 3^{m_3},\dots, i^{m_i},\dots \right>$ is an element of the partially ordered set $\bigoplus_{i\geq 1} \mathbb{N}$, the set of sequences of non-negative integers where only finitely many terms are nonzero. One can establish a partial ordering on the multiplicities by taking $(m_1,m_2,\ldots)\leq (n_1,n_2,\ldots)$ if and only if $m_i\leq n_i$ for every $i\geq 1$. Thus we can form a lattice from the set $\mathcal P$ where the ordering is multiset inclusion (partition divisibility in the sense of \cite{Robert_bracket}); this is a key insight of Andrews' theory of partition ideals. 

This partition lattice is then isomorphic to the lattice of the positive integers ordered by divisibility.
\footnote{The second author presented a preliminary overview of Fig. \ref{figure1} at \cite{MAAIM}.}
The lattice bijection visible in Figure \ref{figure1} unifies the map of Alladi-Erd\H{o}s on partitions into prime parts \cite{AE77}, the lattice-theoretic backdrop of Andrews' theory of partition ideals \cite{A88}, and aspects of the additive-multiplicative theory in \cite{Robert_bracket}.

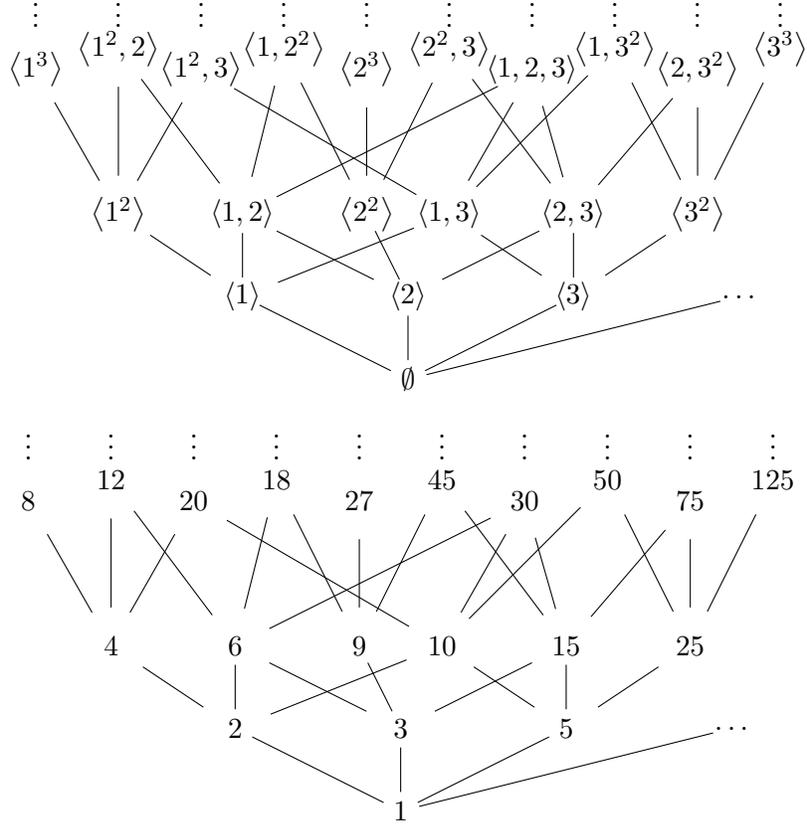
\begin{figure}\label{figure1}
\begin{center}
	    \begin{tikzpicture}
	        \begin{scope}[scale=1.1]
	            \node at (0,0){$\emptyset$};

	            \node at (-2,1){$\left<1\right>$};
	            \node at (0,1){$\left<2\right>$};
	            \node at (2,1){$\left<3\right>$};

	            \node at (-3.5,2){$\left<1^2\right>$};
	            \node at (-2,2){$\left<1,2\right>$};
	            \node at (-.5,2){$\left<2^2\right>$};
	            \node at (.5,2){$\left<1,3\right>$};
	            \node at (2,2){$\left<2,3\right>$};
	            \node at (3.5,2){$\left<3^2\right>$};

	            \node at (-4.5,3.75){$\left<1^3\right>$};
	            \node at (-3.5,4){$\left<1^2,2\right>$};
	            \node at (-2.5,3.75){$\left<1^2,3\right>$};
	            \node at (-1.5,4){$\left<1,2^2\right>$};
	            \node at (-.5,3.75){$\left<2^3\right>$};
	            \node at (0.5,4){$\left<2^2,3\right>$};
	            \node at (1.5,3.75){$\left<1,2,3\right>$};
	            \node at (2.5,4){$\left<1,3^2\right>$};
	            \node at (3.5,3.75){$\left<2,3^2\right>$};
	            \node at (4.5,4){$\left<3^3\right>$};

	            \draw[shorten <=.1in,shorten >=.1in](0,0)--(0,1);
	            \draw[shorten <=.1in,shorten >=.1in](0,0)--(-2,1);
	            \draw[shorten <=.1in,shorten >=.1in](0,0)--(2,1);
	            \draw[shorten <=.1in,shorten >=.1in](0,0)--(4,1);
	            \draw[shorten <=.1in,shorten >=.1in](-2,1)--(-2,2);
	            \draw[shorten <=.1in,shorten >=.1in](0,1)--(-.5,2);
	            \draw[shorten <=.2in,shorten >=.2in](0,1)--(-2,2);
	            \draw[shorten <=.1in,shorten >=.1in](2,1)--(2,2);
	            \draw[shorten <=.2in,shorten >=.2in](-2,1)--(-3.5,2);
	            \draw[shorten <=.2in,shorten >=.2in](-2,1)--(.5,2);
	            \draw[shorten <=.2in,shorten >=.2in](0,1)--(2,2);
	            \draw[shorten <=.2in,shorten >=.2in](2,1)--(.5,2);
	            \draw[shorten <=.2in,shorten >=.2in](2,1)--(3.5,2);

	            \draw[shorten <=.2in,shorten >=.2in](-3.5,2)--(-4.5,3.75);
	            \draw[shorten <=.2in,shorten >=.2in](-3.5,2)--(-3.5,4);
	            \draw[shorten <=.2in,shorten >=.2in](-3.5,2)--(-2.5,3.75);

	            \draw[shorten <=.2in,shorten >=.2in](-2,2)--(-3.5,4);
	            \draw[shorten <=.2in,shorten >=.2in](-2,2)--(-1.5,4);
	            \draw[shorten <=.2in,shorten >=.2in](-2,2)--(1.5,3.75);

	            \draw[shorten <=.2in,shorten >=.2in](-.5,2)--(-1.5,4);
	            \draw[shorten <=.2in,shorten >=.2in](-.5,2)--(-.5,3.75);
	            \draw[shorten <=.2in,shorten >=.2in](-.5,2)--(.5,4);

	            \draw[shorten <=.2in,shorten >=.2in](.5,2)--(-2.5,3.75);
	            \draw[shorten <=.2in,shorten >=.2in](.5,2)--(1.5,3.75);
	            \draw[shorten <=.2in,shorten >=.2in](.5,2)--(2.5,4);

	            \draw[shorten <=.2in,shorten >=.2in](2,2)--(.5,4);
	            \draw[shorten <=.2in,shorten >=.2in](2,2)--(1.5,3.75);
	            \draw[shorten <=.2in,shorten >=.2in](2,2)--(3.5,3.75);

	            \draw[shorten <=.2in,shorten >=.2in](3.5,2)--(2.5,4);
	            \draw[shorten <=.2in,shorten >=.2in](3.5,2)--(3.5,3.75);
	            \draw[shorten <=.2in,shorten >=.2in](3.5,2)--(4.5,4);

	            \node at (4,1){$\ldots$};

	            \node at (-4.5,4.5){$\vdots$};
	            \node at (-3.5,4.5){$\vdots$};
	            \node at (-2.5,4.5){$\vdots$};
	            \node at (-1.5,4.5){$\vdots$};
	            \node at (-.5,4.5){$\vdots$};
	            \node at (4.5,4.5){$\vdots$};
	            \node at (3.5,4.5){$\vdots$};
	            \node at (2.5,4.5){$\vdots$};
	            \node at (1.5,4.5){$\vdots$};
	            \node at (.5,4.5){$\vdots$};

	        \end{scope}
	    \end{tikzpicture}
	\end{center}

		\begin{center}
	    \begin{tikzpicture}
	        \begin{scope}[scale=1.1]
	            \node at (0,0){1};

	            \node at (-2,1){2};
	            \node at (0,1){3};
	            \node at (2,1){5};

	            \node at (-3.5,2){4};
	            \node at (-2,2){6};
	            \node at (-.5,2){9};
	            \node at (.5,2){10};
	            \node at (2,2){15};
	            \node at (3.5,2){25};

	            \node at (-4.5,3.75){8};
	            \node at (-3.5,4){12};
	            \node at (-2.5,3.75){20};
	            \node at (-1.5,4){18};
	            \node at (-.5,3.75){27};
	            \node at (0.5,4){45};
	            \node at (1.5,3.75){30};
	            \node at (2.5,4){50};
	            \node at (3.5,3.75){75};
	            \node at (4.5,4){125};

	            \draw[shorten <=.1in,shorten >=.1in](0,0)--(0,1);
	            \draw[shorten <=.1in,shorten >=.1in](0,0)--(-2,1);
	            \draw[shorten <=.1in,shorten >=.1in](0,0)--(2,1);
	            \draw[shorten <=.1in,shorten >=.1in](0,0)--(4,1);
	            \draw[shorten <=.1in,shorten >=.1in](-2,1)--(-2,2);
	            \draw[shorten <=.1in,shorten >=.1in](0,1)--(-.5,2);
	            \draw[shorten <=.2in,shorten >=.2in](0,1)--(-2,2);
	            \draw[shorten <=.1in,shorten >=.1in](2,1)--(2,2);
	            \draw[shorten <=.2in,shorten >=.2in](-2,1)--(-3.5,2);
	            \draw[shorten <=.2in,shorten >=.2in](-2,1)--(.5,2);
	            \draw[shorten <=.2in,shorten >=.2in](0,1)--(2,2);
	            \draw[shorten <=.2in,shorten >=.2in](2,1)--(.5,2);
	            \draw[shorten <=.2in,shorten >=.2in](2,1)--(3.5,2);

	            \draw[shorten <=.2in,shorten >=.2in](-3.5,2)--(-4.5,3.75);
	            \draw[shorten <=.2in,shorten >=.2in](-3.5,2)--(-3.5,4);
	            \draw[shorten <=.2in,shorten >=.2in](-3.5,2)--(-2.5,3.75);

	            \draw[shorten <=.2in,shorten >=.2in](-2,2)--(-3.5,4);
	            \draw[shorten <=.2in,shorten >=.2in](-2,2)--(-1.5,4);
	            \draw[shorten <=.2in,shorten >=.2in](-2,2)--(1.5,3.75);

	            \draw[shorten <=.2in,shorten >=.2in](-.5,2)--(-1.5,4);
	            \draw[shorten <=.2in,shorten >=.2in](-.5,2)--(-.5,3.75);
	            \draw[shorten <=.2in,shorten >=.2in](-.5,2)--(.5,4);

	            \draw[shorten <=.2in,shorten >=.2in](.5,2)--(-2.5,3.75);
	            \draw[shorten <=.2in,shorten >=.2in](.5,2)--(1.5,3.75);
	            \draw[shorten <=.2in,shorten >=.2in](.5,2)--(2.5,4);

	            \draw[shorten <=.2in,shorten >=.2in](2,2)--(.5,4);
	            \draw[shorten <=.2in,shorten >=.2in](2,2)--(1.5,3.75);
	            \draw[shorten <=.2in,shorten >=.2in](2,2)--(3.5,3.75);

	            \draw[shorten <=.2in,shorten >=.2in](3.5,2)--(2.5,4);
	            \draw[shorten <=.2in,shorten >=.2in](3.5,2)--(3.5,3.75);
	            \draw[shorten <=.2in,shorten >=.2in](3.5,2)--(4.5,4);

	            \node at (4,1){$\ldots$};

	            \node at (-4.5,4.5){$\vdots$};
	            \node at (-3.5,4.5){$\vdots$};
	            \node at (-2.5,4.5){$\vdots$};
	            \node at (-1.5,4.5){$\vdots$};
	            \node at (-.5,4.5){$\vdots$};
	            \node at (4.5,4.5){$\vdots$};
	            \node at (3.5,4.5){$\vdots$};
	            \node at (2.5,4.5){$\vdots$};
	            \node at (1.5,4.5){$\vdots$};
	            \node at (.5,4.5){$\vdots$};

	        \end{scope}
	    \end{tikzpicture}
	\end{center}
	\caption{Illustration of two lattices: partitions ordered by multiset inclusion, integers ordered by divisibility}
\end{figure}

We note that the index statistic $\operatorname{idx}(p), p\in\mathbb P,$ introduced in the proof above of Proposition \ref{prop0} highlights another bijection, between natural numbers and primes, which is implicit in Figure \ref{figure1}. 
Let us define 
the {\it pre-index} $\operatorname{pdx}(i)\in\mathbb P$ of $i\in \mathbb N$ by the prime number $p\in \mathbb P$ whose index $\operatorname{idx}(p)$ is equal to $i$; i.e., $\operatorname{pdx}(i):=p_i$, and we define $\operatorname{pdx}(0):=1$. Then 
%
%
$$\left(\operatorname{pdx}\circ \operatorname{idx}\right)(p)=p,\  \  \  \  \  \  \  \  \left(\operatorname{idx}\circ \operatorname{pdx}\right)(i)=i.$$
\begin{remark}
Furthermore, one has the pair of relations
$${N}(\lambda)=\prod_{p\in\mathbb P}\operatorname{idx}(p)^{m_{\operatorname{idx}(p)}(\lambda)},\  \  \  \  \  \  \  \  \widehat{N}(\lambda)=\prod_{i\in \mathbb N}\operatorname{pdx}(i)^{m_i(\lambda)},$$
where $m_j(\lambda)$ is the multiplicity of $j\geq 0$ in the partition $\lambda$, as above.
\end{remark} 
%

Table \ref{table1} gives a dictionary 
between partition-theoretic objects and their arithmetic counterparts consistent with these bijective maps; the functions $p_{\operatorname{min}}(n), p_{\operatorname{max}}(n)\in\mathbb P$ denote the {\it smallest prime factor} and {\it largest prime factor} of $n\in \mathbb N$, respectively. The right-hand column of the table gives translations from the left column using ``supernormal'' mappings.\footnote{In Section \ref{sect5} we discuss Table \ref{table1} further, and use it to formulate new conjectures.}


\begin{table}[h]
    \centering
    \begin{tabular}{| c | c | c | }
\hline
 \  & {\bf Partition side} &  {\bf Arithmetic side}\\ \hline
 1&$\mathcal P$&$\mathbb N$ \\ \hline
 2&$\lambda$&$n$\  \  \  (supernorm)  \\ \hline
 3&$\lambda\cdot \lambda'$&$n\cdot n'$  \\ \hline
 4&Subpartitions of $\lambda$ &Divisors of $n$ \\ \hline
5&Parts of $\lambda$& Prime factors of $n$\  \  \  (indices)\\ \hline
6&$\left<1^{m_1}, 2^{m_2}, \dots, i^{m_i}, \dots \right>$& $p_1^{m_1}p_2^{m_2}\cdots p_i^{m_i}\cdots$\\ \hline
 7&$\ell(\lambda)$ & $\Omega(n)$,\  number of prime factors \\ \hline
 8&$\rm{sm}(\lambda)$  & $p_{\rm{min}}(n)$\  \  \  (index) \\ \hline
9&$\rm{lg}(\lambda)$&  $p_{\rm{max}}(n)$\  \  \  (index)\\ \hline
10&Partitions with no $1$'s & Odd integers \\ \hline
11&Partitions containing $1$'s & Even integers \\ \hline
12&Partitions into distinct parts &Square-free integers \\ \hline
13&$\mu_{\mathcal{P}}(\lambda)$&$\mu(n)$ \\ \hline
14&$q$-series & Dirichlet series \\ \hline
15&$\prod_{i=1}^{\infty}(1-q^i)$ & ${\zeta(s)^{-1}}$\\  \hline
16&$q^i$ & $p_i^{-s}$ \\ \hline
17&$q^{\vert \lambda \vert}$\  \  in generating functions  
& $n^{-s}$\  \  in Dirichlet series
\\ \hline
18&$q\rightarrow 1^-$ & $s\to1^+$ \\\hline
19&$\sum_i a_i q^i$ & $\sum_i a_i p_i^{-s}$ \\ \hline
20&$(1-q)^{-1}$ & $1+\sum_{p\in \mathbb P} p^{-s}$ \\ 
\hline
\end{tabular}\vskip.1in
    \caption{Dictionary between partition theory and arithmetic.} 
    \label{table1}
\end{table}
%

Before we give applications of this isomorphism, in Section \ref{sectanalytic} we prove analytic facts relating these mappings to other partition-theoretic statistics. In Section \ref{sect3}, 
we prove analytic connections between the partition 
norm, size, and length. In Section \ref{sect4}, we give a similar analysis related to the supernorm. 

In Section \ref{sect2}, from consideration of Figure \ref{figure1}, we 
ask: {what can we say about the arithmetic densities of subsets of integers using ``supernormal'' maps?} This leads us to prove a new identity giving arithmetic densities of subsets of $\mathbb N$ as an analogue of a formula of Kural-McDonald-Sah \cite{K20}. 

In Section \ref{sect5}, as a demonstration of how one might apply the correspondences in Table \ref{table1}, we use a formula of Frobenius and a theorem from \cite{Robert_abelian} to formulate new conjectures that include a generalization of the main arithmetic density theorem of Section \ref{sect2}.

\section{Analytic comparisons of partition statistics}\label{sectanalytic}

 Since $p_i \sim i \log i$ as $i\rightarrow\infty$ by the prime number theorem, equation \eqref{supernormdef} suggests the rough approximation\footnote{By ``$\approx$'' we mean approximately equal, which is to say, of comparable magnitudes.}
$$\widehat{N}(\lambda)\  \approx\ 2^{m_1(\lambda)}\prod_{i\geq 2}(i \log i)^{m_i(\lambda)}\  \approx\  N(\lambda) \cdot 2^{m_1(\lambda)}\prod_{i\geq 2}(\log i)^{m_i(\lambda)},$$
where we pulled out the $i=1$ approximating factor to avoid getting zeroes. 
This approximation has considerable error at smaller indices $i\geq 2$, as we show in Section \ref{sect4} below, but it does qualitatively display the following: {\it partitions with greater lengths have supernorm a lot larger than their norm} -- and thus potentially {\it much} larger than their length and size, since usually
\begin{equation}\label{ineq2} \ell(\lambda)\  \leq\   |\lambda|\  \leq\  N(\lambda)\  \leq \widehat{N}(\lambda),\end{equation} 
noting the second inequality is violated if $\lambda$ has a relatively large number $m_1(\lambda)>N(\lambda)-|\lambda|$ of $1$'s as parts. On the other hand, {partitions with fewer parts have supernorm closer to their norm}.

In this section, we study the relative magnitudes of these partition statistics.\footnote{Studies of other, related connections are found in works of Kumar and Rana \cite{KumarRana_supernorm2, KumarRana_supernorm}.} 

\subsection{Comparing $\ell(\lambda)$, $|\lambda|$, and $N(\lambda)$}\label{sect3}
\subsubsection{Explicit formula for $\log N(\lambda)$}
In this subsection and the next, we examine interrelations and asymptotics between the length, size, and norm. First we prove a formula for the natural logarithm of the norm, written in terms of the length and largest part statistics.

For a partition $\lambda=\left\langle1^{m_1},2^{m_2},\dots, r^{m_r}\right\rangle$ with $\operatorname{lg}(\lambda)=r$, define $\lambda(m,n)$ to be the subpartition of $\lambda$ containing only the parts of $\lambda$ with sizes $m,m+1, m+2,\dots, n-1, n$.
\begin{thm}\label{thm2}
Let $r>1$ be an integer, and define $\varepsilon_r$ to be 1 if $r$ is odd or 0 if $r$ is even.  For a partition $\lambda=\left\langle1^{m_1},2^{m_2},\dots, r^{m_r}\right\rangle$, we have that
\begin{align*}
\log N(\lambda)&=\ell(\lambda)\log r+\varepsilon_rm_{\left\lfloor\frac{r}{2}\right\rfloor}\log\left(\left\lfloor\frac{r}{2}\right\rfloor\right)\\
&\hspace{1cm}+\sum_{i=2}^{\left\lfloor\frac{r}{2}\right\rfloor}\big[\ell(\lambda(i,r-i+1))-\ell(\lambda(1,r-i))\big]\log(r-i+1).
\end{align*}
\end{thm}

\begin{proof}
Applying Abel summation to $\log N(\lambda)=\sum\limits_{i=1}^r m_i\log i$ yields
\begin{flalign*}
\sum_{i=1}^rm_i\log i&=\sum_{i=1}^rm_i\log r-\int_1^r\left(\frac{1}{u}\sum_{i=1}^um_i\right)du\\ &=\sum_{i=1}^rm_i\log r-\int_1^r\frac{\ell(\lambda(1,u))}{u}\,du.
\end{flalign*}
In other words, $\ell(\lambda(1,u))$ is the number of parts in $\lambda$ with size up to $u$.  Since $\ell(\lambda(1,u))$ is a stair-step function which only changes value at each integer value of $u$, we can break up the integral into $r-1$ separate integrals:
\begin{flalign*}&\int_1^r\frac{\ell(\lambda(1,u))}{u}\,du\\ 
&=\int_1^2\frac{\ell(\lambda(1,1))}{u}\,du+\int_2^3\frac{\ell(\lambda(1,2))}{u}\,du+\cdots+\int_{r-1}^r\frac{\ell(\lambda(1,r-1))}{u}\,du.\end{flalign*}
The lengths of the subpartitions are now constant on each interval, so we can evaluate each integral separately, viz.
\begin{align*}
\int_1^2\frac{\ell(\lambda(1,1))}{u}\,du&=\ell(\lambda(1,1))\log2,\\
\int_2^3\frac{\ell(\lambda(1,2))}{u}\,du&=\ell(\lambda(1,2))\big(\log3-\log2\big),\\
&\vdots\\
\int_{r-1}^r\frac{\ell(\lambda(1,r-1))}{u}\,du&=\ell(\lambda(1,r-1))\big(\log r-\log(r-1)\big).
\end{align*}
The sum of these integrals simplifies when we rewrite $\ell(\lambda(1,n))=m_1+m_2+\cdots +m_n$ for each $n$:
\begin{align*}
\int_1^2\frac{\ell(\lambda(1,1))}{u}\,du &+\cdots+\int_{r-1}^r\frac{\ell(\lambda(1,r-1))}{u}\,du\\
&=-m_2\log2-m_3\log3-\cdots-m_{r-1}\log(r-1)\\
&\hspace{1cm}+m_1\log r+m_2\log r+\cdots+m_{r-1}\log r\\
&=\ell(\lambda(1,r-1))\log r-\sum_{i=2}^{r-1}m_i\log i.
\end{align*}
Now, we can write $\log N(\lambda)$ as follows:
$$\sum_{i=1}^rm_i\log i=\ell(\lambda)\log r-\ell(\lambda(1,r-1))\log r+\sum_{i=2}^{r-1}m_i\log i.$$
The remaining sum is $\log N(\lambda(2,r-1))$, which we can rewrite by using Abel summation again:
$$\sum_{i=2}^{r-1}m_i\log i=\big[\ell(\lambda(2,r-1))-\ell(\lambda(1,r-2))\big]\log(r-1)+\sum_{i=3}^{r-2}m_i\log i.$$
More generally, the $j$th remaining sum, $1\leq j\leq\left\lfloor\frac{r}{2}\right\rfloor$, can be written as follows:
$$\sum_{i=j}^{r-j+1}m_i\log i=\big[\ell(\lambda(j,r-j+1))-\ell(\lambda,1,r-j))\big]\log(r-j+1)+\sum_{i=j+1}^{r-j}m_i\log i.$$
Repeating Abel summation on the remaining sum each time until the bounds meet at $\left\lfloor\frac{r}{2}\right\rfloor$ yields the desired formula for $\log N(\lambda)$.
\end{proof}

\subsubsection{Explicit formula for $|\lambda|$}
In this subsection we record an explicit formula for the partition size, written in terms of the length and largest part.
\begin{thm}\label{thm3}
For partition $\lambda=\left\langle1^{m_1},2^{m_2},\dots,r^{m_r}\right\rangle$, we have that $$|\lambda|=\ell(\lambda)r+\sum_{i=3}^r\big(r-i\big)\ell(\lambda(1,i-1))-rm_1.$$
\end{thm}

\begin{proof}
Applying Abel summation to $|\lambda|=\sum\limits_{i=1}^rm_i\cdot i$ yields
\begin{align*}
\sum_{i=1}^rm_i\cdot i&=\sum_{i=1}^rm_i\cdot r-m_1-\int_1^r\left(\sum_{i=1}^um_i\right)du\\
&=\ell(\lambda)r-m_1-\int_1^r\ell(\lambda(1,u))\,du.
\end{align*}
Again using the fact that $\ell(\lambda(1,u))$ is a stair-step function which only changes value at the integers, we can break up the integral into $r-1$ integrals and evaluate each integral separately:
\begin{align*}
&\int_1^r\ell(\lambda(1,u))\,du\\&=\int_1^2\ell(\lambda(1,1))\,du+\int_2^3\ell(\lambda(1,2))\,du+\cdots+\int_{r-1}^r\ell(\lambda(1,r-1))\,du\\
&=\ell(\lambda(1,1))+\ell(\lambda(1,2))+\cdots+\ell(\lambda(1,k-1))=\sum_{i=1}^{r-1}\ell(\lambda(1,i)).
\end{align*}
Rewriting each subpartition length as a sum of multiplicities, we obtain the following:
\begin{align*}
\sum_{i=1}^{r-1}\ell(\lambda(1,i))&=\sum_{i=1}^{r-1}\sum_{j=1}^im_j\\
&=(r-1)m_1+(r-2)m_2+\cdots+2m_{r-2}+m_{r-1}\\
&=\sum_{i=1}^{r-1}(r-i)m_i\\
&=\sum_{i=1}^{r-1}rm_i-\sum_{i=1}^{r-1}m_i\cdot i\\
&=\ell(\lambda(1,r-1))r-\sum_{i=1}^{r-1}m_i\cdot i.
\end{align*}
Then we obtain the following formula for $|\lambda|$:
$$\sum_{i=1}^rm_i\cdot i=\ell(\lambda)r-m_1-\ell(\lambda(1,r-1))r+\sum_{i=1}^{r-1}m_i\cdot i.$$
Repeating Abel summation on the remaining sum, we have that
$$\sum_{i=1}^{r-1}m_i\cdot i=\ell(\lambda(1,r-1))r-m_1-(r-1)\ell(\lambda(1,r-2))+\sum_{i=1}^{r-2}m_i\cdot i.$$
In general, applying Abel summation to the remaining sum up to $j$ yields the following formula:
$$\sum_{i=1}^jm_i\cdot i=\ell(\lambda(1,j))r-m_1-j\ell(\lambda(1,j-1))+\sum_{i=1}^{j-1}m_i\cdot i.$$
Therefore, the formula for $|\lambda|$ simplifies to the following:
\begin{align*}
&\sum_{i=1}^rm_i\cdot i=\ell(\lambda) r-\ell(\lambda(1,r-1))r\\
&\hspace{2cm}+\ell(\lambda(1,r-1))r-(r-1)\ell(\lambda(1,r-2))\\
&\hspace{2cm}+\cdots+\ell(\lambda(1,2))r-2\ell(\lambda(1,1))+m_1-(r-1)m_1\\
&=\sum_{i=1}^{r-1}\big[\ell(\lambda(1,r-i+1))r-(r-i+1)\ell(\lambda(1,r-i))\big]-(r-2)m_1\\
&=r\sum_{i=2}^r\ell(\lambda(1,i))-\sum_{i=1}^{r-1}(i+1)\ell(\lambda(1,i))-(r-2)m_1.
\end{align*}
This gives the desired result.
\end{proof}

\subsection{Analysis of the supernorm}\label{sect4}

\subsubsection{Explicit bounds} In this subsection and the next, we look at estimating the magnitude of the supernorm, and finish with a combinatorial connection between the length, size, norm and supernorm statistics.

 \begin{thm}\label{normestimatethm}
        For $\lambda\in \mathcal P$ we have 
        \[\widehat{N}(\lambda)\leq N(\lambda)^{\frac{\log 3}{\log 2}}.\]
    \end{thm}

    \begin{proof}
        If $\lambda = \langle 1^{m_1},2^{m_2},\ldots \rangle$ and we set $N_d:= d^{m_d}$, then a simple computation shows \[ \widehat{N}(\lambda) = \prod_d N_d^{\frac{\log p_d}{\log d}} \] where $p_d$ is the $d$th prime. We now focus on the exponents \[a(d):=\frac{\log p_d}{\log d}\] and note that it suffices to show $a(d) \leq \log 3 / \log 2$ for all $d$. This can be shown quickly using the prime number theorem, but we choose to give a more elementary proof. One can easily obtain the upper bound for $p_d$: \[ p_d \leq 2^{e/2} d \log d, \] and we can use this to obtain an upper bound for $a(d)$: \[a(d) \leq 1+ \frac{e \log 2}{2 \log d} + \frac{\log_2 d}{\log d}.\] Now we see the inequality $a(d) \leq \log 3/\log 2$ certainly holds for $d>52$; the remaining cases can quickly be checked. \end{proof}

        \begin{thm}\label{supernorminterval}
    If $\lambda$ is a partition with $|\lambda|=n$ and $p_n$ is the $n$th prime, then \[ p_n\leq \widehat{N}(\lambda) \leq 2^n.\]
\end{thm}

\begin{proof}
    We start with the upper bound. Suppose that among all partitions of size $n$, $\lambda$ maximizes $\widehat{N}(\lambda)$. Then it suffices to show that $\ell(\lambda)=n$, since in this case $\lambda = \left\langle 1^n \right\rangle$ and $\widehat{N}(\lambda)=2^n$. Suppose then that $\ell(\lambda)=k<n$. There must be a part of $\lambda$ that is larger than 1, so let $r$ denote a largest such part. We can then inductively assume that the partition is of the form $\lambda =\left<1^{n-r}, r\right>$ (for $n=1$, $n=2$, etc. the hypothesis is easily verified). Now we have that $\widehat{N}(\left<1^{n-r}, r\right>)=2^{n-r}p_r$, but a simple application of Bertrand's postulate shows that $p_r<2 p_{r-1}$. This shows that as long as $k<n$, we can increase $k$ while only increasing the integer to which our partition is mapped by $\widehat{N}$; in other words, we must have $k=n$.

    For the lower bound, we argue similarly except now assuming that $\lambda$ minimizes $\widehat{N}(\lambda)$ among all partitions of size $n$. It suffices to show that $\ell(\lambda)=1$, and we notice that for $n=1$, $n=2$, $n=3$, etc. this is the case. Thus we can inductively assume that $\widehat{N}(\lambda)=p_rp_{n-r}$. It now suffices to show that $p_a p_b > p_{a+b}$ for all positive integers $a$ and $b$. Notice that if $a=1$ this is simply Bertrand's postulate, so assume without loss of generality that $1<a\leq b$. Thus $p_{a+b} \leq p_{2b}$ and $p_a p_b \geq 3p_b$. We now show that $p_{2b} \leq 3 p_b$ for all $b$. By a result of Rosser and Shoenfeld \cite{ros}, we have for $n\geq 2$ \[ p_n > n (\log n + \log\log n -3/2)\] and for $n\geq 20$ \[p_n < n ( \log n + \log\log n - 1/2).\] A simple computation shows that $p_{2b} \leq 3p_{b}$ for all $b\geq 48$. The remaining cases can be quickly verified with a computer. \end{proof}

    \subsubsection{Other considerations}

  Since the supernorm of a partition shares multiplicative similarities to the norm, we might hope to develop formulas like those described in Theorems \ref{thm2} and \ref{thm3}. Let $\lambda=\langle 1^{m_1},2^{m_2},\ldots \rangle$. Then
\[\log \widehat{N}(\lambda) = \sum_{i\geq 1} m_i \log p_i,\]
and it is reasonable to try to apply the prime number theorem to estimate $\log p_i$. The trouble is that most numbers will have small prime factors for which the prime number theorem is not accurate. For example, $p_3=5$, while $3\log 3 \approx 3.30$. If the part 3 appears in the partition $\lambda$ multiple times, then the error in this approximation compounds. 

However, it is possible to explicitly relate the length, size, norm, and supernorm for special classes of partitions. Here we give an example of such a class, and we note that other related classes could be studied as well. 
%
%
%
%

    Let $\lambda_{[n]}:=\left<1,2,3,\dots,n-1,n \right>\in \mathcal P$ denote the partition into distinct parts consisting of the first $n$ natural numbers. Then we see immediately\footnote{For $p,p'\in \mathbb P$, recall that the $p$-{\it primorial} is defined by $p\#:=\prod_{p'\leq p} p'$ (product of primes $\leq p$).}:
    \begin{align*}
    &\ell\left(\lambda_{[n]}\right)\  =\  n\  =\  \exp(\log n)\text{ as $n\rightarrow \infty$},\\
        &\left|\lambda_{[n]}\right|\  =\  {n+1\choose 2} =\  \exp(2\log n + O(1)) \text{ as $n\rightarrow \infty$},\\
        &N\left(\lambda_{[n]}\right)\  =\   n! \  =\  \exp(n\log n - n + o(n))\text{ as $n\rightarrow \infty$},\\
       &\widehat{N}\left(\lambda_{[n]}\right)\  =\  p_n\# \  =\  \exp(n\log n (1+o(1)))\text{ as $n\rightarrow \infty$}.
    \end{align*}
Interestingly, the left-hand sides give natural partition-theoretic interpretations for the classical combinatorial statistics and asymptotics on the right.  

\section{Partition bijections and arithmetic densities}\label{sect2}

One point that is evident when considering the isomorphism between $\mathcal P$ and $\mathbb N$ in Figure \ref{figure1} is that {bijections between special subsets of partitions imply analogous bijections between subsets of natural numbers}. 
Now, recall that the {\it arithmetic density} $d(S)$ of a subset $S\subseteq \mathbb N$ is defined by 
        \begin{equation}\label{dS} d(S):=\lim_{N\to \infty}\frac{\#\{n\in S\  :\  n\leq N\}}{N},\end{equation}
if the limit exists. 
Then it is natural to wonder: {\it what can we say about the arithmetic densities of subsets of integers using ``supernormal'' maps?} 

\subsection{Comparing arithmetic densities of bijective subsets}

Note that partition conjugation yields the well-known bijection between partitions of {\it length} $k$ and partitions with {\it largest part} $k$. Within the set $\mathbb N$, since parts of partitions map to indices of prime factors in the ``supernormal'' setting, 
this bijection induces a bijection between classes of integers with restricted prime factorizations.

\begin{example}\label{ex1}
There is a natural bijection between integers with $k\geq 1$ prime factors including repetition, and integers whose largest prime factor is $p_k$ (the $k$th prime). Furthermore, among the integers whose prime factors' indices sum to fixed $n\geq 1$, the number of integers with $k$ prime factors is equal to the number of integers with largest prime factor $p_k$. 
\end{example}

\begin{remark}
Note that the sum of the indices $\sum_{p|n}\operatorname{idx}(p)$ of the prime factors of an integer $n$, including repetition, is analogous to partition size in this context.
\end{remark}
Similarly, the well-known bijection of Euler between partitions into {\it distinct parts} and partitions into  {\it odd parts} implies another fact about integers. 

\begin{example}\label{ex2}
There is a natural bijection between square-free integers and integers with only odd-indexed prime factors, e.g. $p_1=2, p_3=5, p_5=11,$ etc. Furthermore, among the integers whose prime factors' indices sum to fixed $n\geq 1$, the number of squarefree integers is equal to the number of integers with only odd-indexed prime factors. \end{example}

\begin{remark}
The general partition result gives the same bijection between $k$th power-free integers, and integers with no prime factor's index divisible by $k$.
\end{remark}

%
%
%
%
%
%
%
%
These bijections were given in \cite{KM} from the perspective of multiplicative partitions. Such obscure properties of integers would be difficult to prove (or even notice) without ideas about partitions. Likewise, many theorems about partitions can be reinterpreted as facts about special subsets of $\mathbb N$.

A noteworthy feature of Example \ref{ex1} 
is that the arithmetic densities of these two bijective subsets are equal: 
almost all natural numbers do {\it not} have exactly $k$ prime factors; it is just as clear that almost all natural numbers do {\it not} have largest prime factor $p_k$ for fixed $k$. Both of the subsets of $\mathbb N$ in Example \ref{ex1} can be shown to have arithmetic density zero. 

Is it always the case that bijective subsets arising from the supernorm map also have matching arithmetic densities? 
Let us look at Example \ref{ex2} above: the density of squarefree numbers is well known to be $6/\pi^2$. Is this also the density of the set $S_k$ of integers with no prime factor's index divisible by $k$?

 \begin{proposition}\label{prop1}
                Let $S_k\subset \mathbb N$ be the subset of natural numbers with no prime factor's index divisible by fixed $k\geq 2$. Then \[d(S_k)  = 0.\]
            \end{proposition}
            
            \begin{proof}[Proof of Proposition \ref{prop1}]
                We use a straightforward sieving argument. Let $\mathbb P_k$ be the set of all primes that {\it are} indexed by a multiple of $k$; i.e., $p_j\in\mathbb P_k$ if and only if $k\  |\  j$. Then if we choose a parameter $T=T(x)$ depending on $x\in \mathbb R^+$ and define \[\mathbb P_k(T):=\mathbb P_k\cap[2,T]\] and $$S_k(x):=\#\left\{n\in S_k:n\leq x\right\},$$ we can write \[0\leq S_k(x) \leq x- \sum_{j\geq 1}\sum_{\substack{A\subset \mathbb P_k(T) \\ |A|=j}}(-1)^j\left\lfloor \frac{x}{\prod_{p\in A} p}\right\rfloor,\] which is an application of the principle of inclusion-exclusion, with $\left\lfloor X \right\rfloor$ the floor function of $X\in \mathbb R$. We can break up the floor function using $\left\lfloor X \right\rfloor = X-\{X\}$, with $\{X\}$ the fractional part of $X\in \mathbb R$,  to see that \[S_k(x) \leq x\prod_{p\in \mathbb P_k(T)}\left(1-\frac{1}{p}\right) + O\left(2^{\pi(T)}\right), \] where $\pi(X)$ is the total number of primes not exceeding $X\geq 2$.  We recall that $\pi(X)\sim X/\log X$ as $X\to \infty$ by the prime number theorem. Setting $T=\log x$ gives \[S_k(x) \leq x \prod_{p\in \mathbb P_k(\log x)} \left(1-\frac{1}{p}\right) + o(x).\] This gives the upper bound on the arithmetic density: 
                \begin{equation}\label{prod} d(S_k)\  =\  \lim_{x\to \infty}\frac{S_k(x)}{x}\   \leq\   \prod_{p\in \mathbb P_k} \left(1-\frac{1}{p}\right).\end{equation} 
                That the product diverges to zero is equivalent to the divergence of the sum \begin{equation*}\sum_{p\in \mathbb P_k} \frac{1}{p},\end{equation*} which is clear since in the tail of the series, $p_{mk} \sim k\cdot m\log m$ as $m\rightarrow \infty$  by the prime number theorem. Thus $d(S_k)$ is zero. \end{proof}

So the arithmetic densities of bijective subsets do {not} necessarily match up in the image of the supernorm map, as Example \ref{ex2} demonstrates. As we proved in Section \ref{sect4}, the supernorm scatters partitions of size $n$ over a very large interval in $\mathbb N$ as $n\to \infty$; namely,\begin{equation*} 
p_n\  \leq\  \widehat{N}(\lambda)\   \leq\  2^n\end{equation*} for $|\lambda|=n$. We note that 
$p(n)=o(2^n)$ by the Hardy-Ramanujan asymptotic (see \cite{A88}). 
Seeking explicit transformations between densities of resulting bijective subsets on this large interval would likely be challenging (but potentially worthwhile). 

\subsection{An ``Alladi-like'' formula for arithmetic densities}

A less difficult route to computing arithmetic densities via ``supernormal'' bijections is to map from {\it natural} densities. Recall that the {\it natural density} $d^*(A)$ of a subset $A\subseteq \mathbb P$ is defined by    
\begin{equation}\label{dirdens} d^*(A) = \lim_{x\rightarrow \infty} \frac{\#\{p\in A\  :\  p\leq x\}}{\pi(x)},\end{equation}
if the limit exists, where $\pi(x)$ is the number of primes not exceeding $x\geq 2$.

Following up on a 1977 paper of Alladi \cite{Alladi} in analytic number theory, as well as subsequent work of the first author \cite{dawsey} in algebraic number theory,  
recent works \cite{K20, bwang, Wang3} generalize Alladi's ideas and formulas. These works prove what we will call ``Alladi-like'' formulas of the shape \begin{equation}\label{Alladilike} \sum_{n\geq2} \frac{\mu^*(n)f(n)}{n}\  :=\  \lim_{x\to \infty}\sum_{2\leq n \leq x} \frac{\mu^*(n)f(n)}{n}\  <\  \infty,\end{equation} with $\mu^*(n):=-\mu(n)$ and $f(n)$ an arithmetic function, where the left-hand sum is defined to equal the right-hand limit when the limit exists and converges conditionally in that case.\footnote{Proofs of convergence of such series are given in \cite{Alladi, dawsey}, for example.}  
These types of formulas are used to compute {natural} densities and other constants; for instance, it is a special case of a result of Kural-McDonald-Sah \cite{K20} stated as \cite[Eq. 4]{bwang}, that
        \begin{equation}\label{eqn}
        \sum_{\substack{n\geq2 \\ p_{\operatorname{min}}(n) \in A\subseteq \mathbb P}} \frac{\mu^*(n)}{n}\  = \  d^*(A),
        \end{equation}
        where $A$ is a subset of the {primes} and 
         $p_{\operatorname{min}}(k)$ denotes the {smallest prime factor} of $k\in \mathbb N$. 
         In a related direction, in recent work \cite{OSW, OSW2}, Ono, Wagner, and the third author prove {\it partition-theoretic} formulas to compute {\it arithmetic} densities of subsets of $\mathbb N$ as limiting values of $q$-series as $q\to 1^-$. 
         
As an application of  ``supernormal'' mappings, we will prove an ``Alladi-like'' 
formula analogous to \eqref{eqn} to compute arithmetic densities of subsets $S\subseteq\mathbb N$ (as opposed to natural densities of subsets of $\mathbb P$ like the formulas noted above). 
%
 Recall that we define $\mu^*(n):=-\mu(n)$. 

\begin{thm}\label{thm1}
    Let $S\subseteq \mathbb N$, and let $d(S)$ be the arithmetic density of $S$, if the density exists. 
    Then
    \[\sum_{\substack{n\geq 2 \\ \operatorname{idx}(p_{\operatorname{min}}(n))\in S}} \frac{\mu^*(n)}{n} \  =\   d(S).\]   
\end{thm}

\begin{remark}
The sum is taken over integers for which the {\it index} of the smallest prime factor is in $S\subseteq \mathbb N$.
\end{remark}

\subsection{Proof of the arithmetic density formula}
We use the observation that supernormal maps take the partition $(k)\in\mathcal P$ consisting of the single part $k\geq 1$ to the $k$th prime number; i.e., $\widehat{N}\left((k)\right)=\operatorname{pdx}(k)=p_k$. Then special subsets of $\mathbb P$ are in bijection with special subsets of $\mathbb N$, viz. the prime indices, and have comparable densities. 
%

\begin{lem}\label{lem1}
    Suppose $p\in A\subseteq \mathbb P$ if and only if $\operatorname{idx}(p)\in S\subseteq \mathbb N$; then $d^*(A) = d(S)$, assuming one of the two limits exists.
\end{lem}

\begin{proof}
    Assume that $d^*(A)$ exists. Set $S(x):=\#\{n\in S\  :\  n\leq x\}$. Then for any $\varepsilon>0$ there exists an $N\geq 1$ such that for all $x>N$ we have
    \[d^*(A)-\varepsilon\   <\   \frac{\pi_A(x)}{\pi(x)} \  <\   d^*(A) +\varepsilon,\]  where the we count primes $p\leq x$ with 
    \begin{align*}
        \pi_A(x)\  :=\  \#\{ p \leq x :  p \in A\} \  
         =\  \#\{ n \leq \pi(x) : n \in S \} \  
        =\  S(\pi(x)).
    \end{align*}
    Thus if we choose $M=\pi(N)$, then we have for all $y>M$
    \[d^*(A) - \varepsilon \  <\   \frac{S(y)}{y}\   <\   d^*(A) + \varepsilon,\]
    which shows that $d(S)$ exists and is equal to $d^*(A)$. By symmetry, a similar argument proves that if $d(S)$ exists, then $d^*(A)$ exists and is equal to $d(S)$. \end{proof}

\begin{remark}
Similar arguments might be applied to compare the density of a subset $\{c_{n_1}, c_{n_2}, c_{n_3}, \dots\}\subseteq \{c_1,c_2,c_3,\dots\}=: C$ in any discrete infinite set $C$ to the arithmetic density of the set of indices $\{n_1, n_2, n_3, \dots\}\subseteq \mathbb N$. 
\end{remark}

Applying Lemma \ref{lem1} together with equation \eqref{eqn}, Theorem \ref{thm1} follows.

    \begin{proof}[Proof of Theorem \ref{thm1}]
        We start from equation \eqref{eqn} above. 
        Now let $S$ be any set of positive integers such that $d(S)$ exists, and define
        \begin{equation}\label{Adef} A=A(S):= \left\{ p_a \in \mathbb P: a\in S \right\}.\end{equation}
        Then by Lemma \ref{lem1}, $d^*(A)$ exists and is equal to $d(S)$. Therefore, by (\ref{eqn}):
        \begin{align*}
            d^*(A)
            =\sum_{\substack{n\geq2 \\ p_{\operatorname{min}}(n) \in A}} \frac{\mu^*(n)}{n}
            =\sum_{\substack{n\geq 2 \\ \operatorname{idx}(p_{\operatorname{min}}(n))\in S}} \frac{\mu^*(n)}{n}\  =\  d(S). 
        \end{align*}
    We note that since $p_i<p_j$ if and only if $i<j$, conditional convergence from \eqref{eqn} continues to hold as we do not alter the order of the summands. Thus the final equality above is the statement of the theorem. \end{proof}

\section{``Supernormal'' maps and Abelian-type conjectures}\label{sect5}

Let us return to considering Table \ref{table1} from Section \ref{sect1}. We note that similar but distinct additive-multiplicative analogies exist in other settings; for instance, there are partition zeta function and arithmetic function analogies proved in \cite{Robert_zeta, Robert_bracket} related to the norm instead of the supernorm, and tables in \cite{OSW, OSW2,SSzeta} give analogies arising in the contexts of those works.

In the citations noted above, the third author and his coauthors used mappings like those in Table \ref{table1} to conjecture, and subsequently prove, theorems on both the additive and multiplicative sides of number theory. Indeed, the present authors used Table \ref{table1}, plus the result \eqref{eqn} of Kural-McDonald-Sah and analogies with \cite{OSW, OSW2}, 
to formulate and prove Theorem \ref{thm1}. 
As examples of this method of formulating conjectures, 
we will state two conjectures that are somewhat speculative but do specialize to known cases.

The first thirteen rows of Table \ref{table1} are direct translations. 
For example, it is clear that $\mu_{\mathcal P}(\lambda)$ in equation \eqref{mobius} transforms to the definition of the classical $\mu(n)$ simply by substituting right-hand-column symbols and words for the corresponding terms from the left column, and vice versa. This theme continues in the subsequent rows, but the correspondences are more analytic in nature. Even so, it is just as clear that, for instance, replacing $q^i$ with $p_i^{-s}$ transforms $\prod_{i=1}^{\infty}(1-q^i)$ into $\prod_{i=1}^{\infty}(1-p_{i}^{-s})=\zeta(s)^{-1}$. Expanding these products yields the mapping between $$\sum_{\lambda\in\mathcal P}\mu_{\mathcal P}(\lambda)q^{|\lambda|}=\sum_{\lambda\in \mathcal P}\mu_{\mathcal P}(\lambda)q^{\lambda_1} q^{\lambda_2}\cdots q^{\lambda_r}$$ and $$\sum_{n\in \mathbb N} \mu(n)n^{-s} = \sum_{\lambda\in \mathcal P}\mu_{\mathcal P}(\lambda) p_{\lambda_1}^{-s}  p_{\lambda_2}^{-s}\cdots p_{\lambda_r}^{-s};$$ 
we note from the product forms that both cases have limiting value zero as $q\to 1^-$ and $s\to 1^+$, respectively.\footnote{Conformal mapping between the complex regions $|q|<1$ and $\operatorname{Re}(s)>1$ gives $s=2/(1+q)\to 1^+$ as $q\to 1^-$.} The last two rows give the analogues of power series and geometric series consistent with this scheme.\footnote{The sum $\sum_p p^{-s}$ is often referred to in the literature as the {\it prime zeta function}; in this setting it is the analogue of the geometric series $q/(1-q)$.}

%
%
%


Now, let us apply the mappings of Table \ref{table1}. 
In a recent paper \cite{Robert_abelian} generalizing \cite{OSW, OSW2}, the third author proves families of Abelian theorems using partition theory and $q$-series. As analogues of this work, we will conjecture Dirichlet series and ``Alladi-like'' Abelian formulas using Table \ref{table1}. 

Let us start with a theorem of Frobenius \cite{Frob} that is a central component of the proofs of \cite{Robert_abelian}. For an arithmetic function $f:\mathbb N \to \mathbb C$ (for which we take $f(0):=0$ when needed), if the limit of the average value of $f$ is
$$L_f:=\lim_{N\to \infty}\  \frac{1}{N}\sum_{k=1}^{N}f(k),$$ then for $|q|<1, q\to 1^-$ {radially},\footnote{In fact, the limit \eqref{Frobeq} holds if $q\to 1$ through any path in a {Stolz sector} of the unit disk; i.e., a region with vertex at $z=1$ such that $\frac{|1-q|}{1-|q|}\leq M$ for some $M> 0$. Can Stolz sectors be mapped to regions in $\operatorname{Re}(s)>1$ to give analogous limiting behavior in $s$?} 
Frobenius gives 
\begin{equation}\label{Frobeq} 
\lim_{q\to 1^-}\  (1-q)\sum_{n\geq 1}f(n)q^n\  =\  L_f.
\end{equation}
If the sequence $\{f(n)\}$ converges, then \eqref{Frobeq} is the statement of {Abel's convergence theorem}. If $\{f(n)\}$ diverges, then one may assign the limiting value $L_f$ to the sequence so long as the {\it average value} converges. 

By Table \ref{table1}, the limit of \eqref{Frobeq} transforms like $\lim_{q\to 1^-}  \mapsto\  \lim_{s\to 1^+}$, and the summands of the power series transform like $f(n)q^n \mapsto f(n) p_n^{-s}=f(\operatorname{idx}(p_n)) p_n^{-s}$. Does the multiplicative factor $1-q=1-q^1$ also transform to $1-p_1^{-s}=1-2^{-s}\to 1/2$ by the $q^i\mapsto p_i^{-s}$ correspondence to yield as the appropriate conjectural analogue of \eqref{Frobeq} the following: 
\begin{equation*}
\lim_{q\to 1^-}\  (1-q)\sum_{n\geq 1}f(n)q^n\  \  \stackrel{?}{\longmapsto}\  \  \lim_{s\to 1^+}\  \frac{1}{2}\sum_{n\geq 1}f(n)p_n^{-s}\  \  ?
\end{equation*}

This is an example of how the analysis, as well as the mappings between partition parts and prime indices, should inform our thinking. We want both sides of our conjectural analogue to have analogous limiting behaviors. 
Ideally, an analogue of \eqref{Frobeq} will mimic the implicit use of L'Hospital's rule as well as the bijective correspondences. Now, the Dirichlet series on the right of the conjectural analogue above also diverges to infinity, but with no divergent denominator term to serve as a counter-balance, it will not yield analogous limiting behavior. To achieve the correct analogue, we want to transform $1-q=1/(1-q)^{-1}$ via the geometric series correspondence in row 20 of Table \ref{table1} instead of using $q^i\mapsto p_i^{-s}$.\footnote{Noting that $1+\sum_p p^{-s} \sim \sum_p p^{-s}$ as $s\to 1^+$.} 


\begin{conjecture}\label{conj1}
For $f(n)$ an arithmetic function, if the limit of the average value $L_f=\lim_{N\to \infty}\frac{1}{N}\sum_{k=1}^{N} f(k)$ exists, then 
$$\lim_{s\to 1^+}\  \frac{\sum_{p\in \mathbb P} f\left(\operatorname{idx}(p)\right)p^{-s}}{\sum_{p\in \mathbb P} p^{-s}}\  \   =\  \  L_f.$$
\end{conjecture}

If $f(n)$ is the indicator function of the set $S\subseteq \mathbb N$, and $A=A(S)\subseteq \mathbb P$ is as defined in \eqref{Adef}, then $L_f=d(S)=d^*(A)$. In this case, Conjecture \ref{conj1} reduces to the definition of the {\it Dirichlet density} of $A\subseteq \mathbb P$, which is well known to equal the natural density $d^*(A)$:
\begin{equation*}\lim_{s\to 1^+}\  \frac{\sum_{p\in A} p^{-s}}{\sum_{p\in \mathbb P} p^{-s}}\  =\  d^*(A).\end{equation*}

\begin{remark} A related conjecture is that $\sum_{n\geq 1}a_n q^n$  $\sim \sum_{n\geq 1} b_n q^n$ as $q\to 1^-$ if and only if $\sum_p a_{\operatorname{idx}(p)}p^{-s}$ $\sim \sum_p b_{\operatorname{idx}(p)}p^{-s}$ as $s\to 1^+$, based on Table \ref{table1}.
\end{remark}

Next we conjecture an ``Alladi-like'' formula in the sense of equation \eqref{Alladilike}, in analogy with the following partition-theoretic $q$-series result from the third author's paper \cite{Robert_abelian}. For $f(n)$ an arithmetic function, setting $\mu_{\mathcal P}^*(\lambda):=-\mu_{\mathcal P}(\lambda)$, it is proved in \cite{Robert_abelian} that 
\begin{flalign}\label{qAbelian}
 \lim_{q\to 1^-}\   \sum_{\lambda \in \mathcal P}\mu_{\mathcal P}^*(\lambda)f\left(\operatorname{sm}(\lambda)\right)q^{|\lambda|}\   \  =\  \  L_f. 
\end{flalign}
Using Table \ref{table1} together with \eqref{qAbelian}, we swap the limits in $q$ and $s$ as before, replace partitions with natural numbers in the indices of the series, and swap $\mu_{\mathcal P}(\lambda) \mapsto \mu(n)$ and $q^{|\lambda|}\mapsto n^{-s}$. 
Finally, $\operatorname{sm}(\lambda)$ maps to the {index} of $p_{\operatorname{min}}(n)$. 

Taking the limit as $s\to 1^+$, if the limit exists, the correspondences of Table \ref{table1} lead us to formulate an ``Alladi-like'' identity.  Note that $\mu^*(n):=-\mu(n)$, $f(\operatorname{idx}(1)):=0$, and $(f\circ \operatorname{idx})(p):=f(\operatorname{idx}(p))$. 


\begin{conjecture}\label{conj2}
For $f(n)$ an arithmetic function, if the limit of the average value $L_f=\lim_{N\to \infty}\frac{1}{N}\sum_{k=1}^{N} f(k)$ exists, then 
\begin{flalign*}
\sum_{n\geq 2}\frac{\mu^*(n) \left(f\circ \operatorname{idx}\right)\left(p_{\operatorname{min}}(n)\right)}{n}\  \  =\  \  L_f.\end{flalign*}
\end{conjecture}
Again, if $f(n)$ is the indicator function of $S\subseteq \mathbb N$, then $L_f=d(S)$ and Conjecture \ref{conj2} reduces to Theorem \ref{thm1}.

    \section*{Acknowledgments}
%

The first author was partially supported by an AMS-Simons Travel Grant and an internal grant from the University of Texas at Tyler. The second author was partially supported by the Research and Training Group grant DMS-1344994 funded by the National Science Foundation.

\end{document}